\documentclass[sn-mathphys,Numbered]{sn-jnl}

\usepackage{graphicx}%
\usepackage{multirow}%
\usepackage{amsmath,amssymb,amsfonts}%
\usepackage{amsthm}%
\usepackage{mathrsfs}%
\usepackage[title]{appendix}%
\usepackage{xcolor}%
\usepackage{textcomp}%
\usepackage{manyfoot}%
\usepackage{booktabs}%
\usepackage{algorithm}%
\usepackage{algorithmicx}%
\usepackage{algpseudocode}%
\usepackage{listings}%
\usepackage{latexsym}
\usepackage{color}
\usepackage{bm}
\usepackage{graphicx}
\usepackage{subfigure}
\usepackage{epstopdf}
\usepackage[misc]{ifsym}

\theoremstyle{thmstyleone}%
\newtheorem{theorem}{Theorem}
\theoremstyle{thmstyletwo}%
\newtheorem{lemma}{Lemma}

\theoremstyle{thmstylethree}%
\newtheorem{definition}{Definition}%
\allowdisplaybreaks[4]

\begin{document}

\title[Linear canonical space-time transform and convolution theorems]{Linear canonical space-time transform and convolution theorems}


\author[1,2]{\fnm{Yi-Qiao} \sur{Xu}}\email{yiqiaoxu1016@163.com}

\author*[1,2]{\fnm{Bing-Zhao} \sur{Li}}\email{li\_bingzhao@bit.edu.cn}

\affil[1]{\orgdiv{School of Mathematics and Statistics}, \orgname{Beijing Institute of Technology},\orgaddress{\city{ Beijing}, \postcode{100081}, \country{China}}}

\affil[2]{\orgdiv{Beijing Key Laboratory on MCAACI}, \orgname{Beijing Institute of Technology},\orgaddress{\city{ Beijing}, \postcode{100081}, \country{China}}}


\abstract{Following the idea of the fractional space-time Fourier transform, a linear canonical space-time transform for 16-dimensional space-time $C\ell_{3,1}$-valued signals is investigated in this paper. First, the definition of the proposed linear canonical space-time transform is given, and some related properties of this transform are obtained. Second, the convolution operator and the corresponding convolution theorem are proposed. Third, the convolution theorem associated with the two-sided linear canonical space-time transform is derived.}




\maketitle

\section{Introduction}
\label{sec1}

Space-time algebra, exemplified by the Clifford $C\ell_{3,1}$ algebra, is an associative algebra developed within the $\mathbb{R}^{3,1}$ space-time vector space equipped with the Minkowski metric \cite{bib1}. Introducing the transforms on space-time algebra is becoming one of the hottest research topics. Hitzer is one of the pioneers who introduce the Fourier transform on space-time algebra. He first explained and formulated the space-time Fourier transform \cite{bib29}. The space-time Fourier transform (SFT) is a hypercomplex non-commutative Fourier transform operating on functions defined on $\mathbb{R}^{3,1}$ and taking values in Clifford's geometric algebra $C\ell_{3,1}$ \cite{bib7, bib8}. Following this direction, other researchers investigated the definitions and properties associated with space-time Fourier transform \cite{bib10, bib23, bib25}. Recently, inspired by the orthogonal splitting of space-time algebra, the Heisenberg's uncertainty principle and the Donoho-Stark uncertainty principle for the space-time Fourier transform  were established successfully \cite{bib5, bib8}.

In order to find the flexibility in representing information in the space-time domain, researchers try to apply the other transforms on space-time algebra. Hitzer generalized the SFT to a special affine space-time Fourier transform (SASFT) for 16-dimensional space-time multivector $C\ell_{3,1}$-valued signals over the domain of space-time $\mathbb{R}^{3,1}$ \cite{bib22}. Zayed proposed a space-time fractional Fourier transform (FrSFT) by generalizing the fractional Fourier transform (FrFT) into the space-time algebra \cite{bib1}. These newly defined transforms are the first in literature to offer analogues of the FrFT and the special affine Fourier transform (SAFT) for Clifford algebras. This opens a new research direction on the theories and applications for new transforms associated with space-time algebra.

The linear canonical transform (LCT) has three free parameters, which can enable the rotation and expansion of the time-frequency plane, this flexibility makes it particularly well-suited for the analysis and processing of non-stationary functions \cite{bib26}. The FT, FrFT, and Fresnel transform are special cases of the LCT. Based on this fact and inspired by the FrSFT and the SASFT, we investigate the linear canonical transform associated with space-time algebra. 

Convolution theory is of great significance in various fields such as mathematics, signal processing, image processing, and optics, especially in the design and implementation of multiplier filters \cite{bib39}. In recent years, great efforts have been made to popularize classical convolution theory and extend its applicability to a wider range of fields. Therefore, many convolution theorems related to linear canonical transform (LCT) have been proposed \cite{bib40, bib41, bib42}. Researchers have also shown significant interest in convolution theory in the field of Clifford algebra. The classical convolution and Mustard convolution of the general two-sided quaternion Fourier transform and the general steerable two-sided Clifford Fourier transform were explored in \cite{bib34, bib35}. Aajaz A. Teali and Firdous A. Shah extended the convolution theorem to Clifford-valued linear canonical transform and two-sided Clifford-valued linear canonical transform \cite{bib37, bib38}. Their contributions provide new tools and methods for signal processing, image processing, and related fields.

The article is structured as follows: Section \ref{sec2} gives an overview of space-time algebra. In section \ref{sec3}, we introduce the concept of linear canonical space-time transform (LCST) and learn its basic properties. Section \ref{sec4} presents the convolution theorem. Section \ref{sec5} presents the convolution theorem for the two-sided LCST. Finally, this article ends with the \ref{sec6} section.

\section{Preliminaries}
\label{sec2}
\subsection{Space-time algebra}
\label{subsec1}

The space-time algebra, is the Clifford algebra $C\ell_{3,1}$ derived from the geometric product of all $\mathbb{R}^{3,1}$ and a 16-dimensional basis algebra:
\begin{equation}
	C\ell_{3,1}=\rm{span}\{1,e_t,e_1,e_2,e_3,e_{12},e_{13},e_{23},e_{t1},e_{t2},e_{t3},i_3,e_{t12},e_{t13},e_{t23},i_{st}\}.\label{eq2}
\end{equation}
It is composed of scalar, vector, bivector, trivector, and pseudoscalar elements.

The associative geometric multiplication of the basis vectors obeys the rule:
\begin{equation}
	e_ie_j+e_je_i=2\epsilon_j\delta(i-j),\qquad i,j\in\{t,1,2,3\}.\label{eq1}
\end{equation}
In the above the Kronecker symbol $\delta(i-j)$ is $1$ if $i = j$, zero otherwise. And $\epsilon_j=-1$, for $j=\{1,2,3\}$, and $\epsilon_j=-1$, for $j = t$.

Here, we will use conventional index notations such as $e_{12}=e_1e_2$, $e_{t12}=e_te_{12}$, etc. Here, $i_3= e_1e_2e_3$ is the unit volume of the 3-vector, $i_{st} = e_ti_3$ is a pseudo-scalar, both squared by $-1$. That is, ${i_3}^2 = {i_ { st}}^2 = -1$.

Each spatial element $h$ can be expressed as a real linear combination of elementary edges given by (\ref{eq2}). This notation allows $h$ to be expressed in the general form:
\begin{equation}
	h=\left\langle h\right\rangle_0+\left\langle h\right\rangle_1+\left\langle h\right\rangle_2+\left\langle h\right\rangle_3+\left\langle h\right\rangle_4,\label{eq3}
\end{equation}
$\left\langle h\right\rangle_0=h_0$, $\left\langle h\right\rangle_1=h_te_t+h_1e_1+h_2e_2+h_3e_3$, $\left\langle h\right\rangle_2=h_{12}e_{12}+h_{13}e_{13}+h_{23}e_{23}+h_{t1}e_{t1}+h_{t2}e_{t2}+h_{t3}e_{t3}$, $\left\langle h\right\rangle_3=h_{123}i_3+h_{t12}e_{t12}+h_{t13}e_{t13}+h_{t23}e_{t23}$, $\left\langle h\right\rangle_4=h_{t123}i_{st}$, with $h_0, h_t, h_1, h_2, h_3, h_{12}, h_{13}, h_{23}, h_{t1}, h_{t2}, h_{t3}, h_{123}, h_{t12}, h_{t13}, h_{t23}, h_{t123} \in \mathbb{R}$. $\left\langle h\right\rangle_0$, $\left\langle h\right\rangle_1$, $\left\langle h\right\rangle_2$, $\left\langle h\right\rangle_3$ and $\left\langle h\right\rangle_4$ are respectively called the 0th degree scalar part, the 1st degree vector part, the 2nd degree bivector part, the 3rd degree three vector part and the pseudo-scalar part.

A useful tool for us will be the following split of space-time multi-vector $h \in C\ell_{3,1}$:
\begin{equation}
	h=h_{+}+h_{-}, \ where\ h_{\pm}=\frac{1}{2}(h \pm e_t\ h\ i_3).\label{eq5}
\end{equation}

It is worth noting that the decomposition (\ref{eq5}) yields the norm identity:
\begin{equation}
	|h|^2=|h_{+}|^2+|h_{-}|^2.\label{eq8}
\end{equation}

We have the following key identities for shifting:
\begin{equation}
	exp\{\alpha e_t\} h_{\pm} exp\{\alpha' i_3\}=h_{\pm} exp\{(\alpha' \mp \alpha) i_3\}=exp\{(\alpha \mp \alpha') e_t\} h_{\pm}, \ \forall \alpha, \alpha' \in \mathbb{R}\label{eq7}
\end{equation}

The principal reverse of $h \in C\ell_{3,1}$ is defined by
\begin{equation}
	\begin{aligned}
		\widetilde{h}&=h_0e_0-h_te_t+\sum_{i=1}^3h_ie_i-h_{12}e_{12}-h_{13}e_{13}-h_{23}e_{23}+\sum_{i=1}^3h_{ti}e_{ti}\\&-h_3i_3+h_{t12}e_{t12}+h_{t13}e_{t13}+h_{t23}e_{t23}-h_{t123}i_{st}.\label{eq9}
	\end{aligned}
\end{equation}

In particular, each space-time element has $h_1$ and $h_2$
\begin{equation}
	\widetilde{h_1+h_2}=\widetilde{h_1}+\widetilde{h_2};\widetilde{h_1h_2}=\widetilde{h_2}\widetilde{h_1}\mathrm{~;~}\widetilde{\widetilde{h_1}}=h_1.\label{eq10}
\end{equation}
The trace of each element $h$ of $C\ell_{3,1}$ is defined as the scalar part of its decomposition, as shown in (\ref{eq3}). That is, it is expressed by the following equation:
\begin{equation}
	Tr(h) \overset{\underset{\mathrm{def}}{}}{=} h_0.\label{eq11}
\end{equation}

The symmetry of cyclic multiplication can be easily verified in the following way:
\begin{equation}
	Tr(h_1h_2h_3)~=~Tr(h_3h_1h_2),~\forall h_1,~h_2,~h_3\in C\ell_{3,1}.\label{eq12}
\end{equation}

Additionally, we can easily prove that
\begin{equation}
	\int_{\mathbb{R}}\mathbf{e}^{ht(x-x_0)}dt=2\pi\delta(x-x_0),\label{eq16}
\end{equation}

Let $h \in C\ell_{3,1}$, $\alpha,\alpha' \in \mathbb{R}$, one has
\begin{equation}
|e_thi_3|=|h|\mathrm{~,~}\left|\mathbf{e}^{\alpha e_t}h\mathbf{e}^{\alpha^{\prime}i_3}\right|=|h|.\label{eq18}
\end{equation}

Based on the linearity and properties (\ref{eq12}) of the traces (\ref{eq11}), the orthogonality of the spatial partitioning can be directly verified in the following sense

Let $f,g \in C\ell_{3,1}, \alpha \in \mathbb{R}$, then
\begin{equation}
Tr\left(\mathbf{e}^{\alpha i_3}\widetilde{f_\pm}g_\mp\right)=0.\label{eq19}
\end{equation}

\subsection{Related works}
\label{subsec2}

Next, we review the definition of the space-time Fourier transform (SFT).

\begin{definition}\label{df1}
	The space-time Fourier transform of $f \in L^1(\mathbb{R}^{3,1},C\ell_{3,1})$ is defined as \cite{bib29}
	\begin{equation}
		\mathcal{F}_{\mathrm{S}FT}[f](\boldsymbol{w})=\int_{\mathbb{R}^{3,1}}\boldsymbol{e}^{-e_tw_tt}f(\boldsymbol{x})\exp\left\{-i_3\underline{x}\cdot\underline{w}\right\}d^4\boldsymbol{x},\label{eq21}
	\end{equation}
	 where $\boldsymbol{w}=w_te_t+\underline{w}\in\mathbb{R}^{3,1},\underline{w}=w_1e_1+w_2e_2+w_3e_3\in\mathbb{R}^3$ is the space-time frequency vector, $\begin{aligned}\boldsymbol{x}&=te_t+\underline{x}\in\mathbb{R}^{3,1},\underline{x}=x_1e_1+x_2e_2+x_3e_3\in\mathbb{R}^3\end{aligned}$ is the space-time vector, and $\underline{x}\cdot\underline{w}= {\textstyle \sum_{k=1}^{k=3}x_kw_k} $ is the natural $\mathbb{R}^3-scalar$ product.
\end{definition}

The inversion formula for the SFT is given by
\begin{equation}
	f(\boldsymbol{x})=\frac1{(2\pi)^4}\int_{\mathbb{R}^{3,1}}\exp{\{e_tw_tt\}}~\mathcal{F}_{SFT}[f](\boldsymbol{w})~\exp{\{i_3\underline{x}\cdot\underline{w}\}}~d^4\boldsymbol{w}.\label{eq22}
\end{equation}

Based on the above research, Zayed first introduced the fractional Fourier transform (FrFT) to space-time algebra \cite{bib1}. Next, we review the definition of the fractional space-time Fourier transform (FrSFT).

\begin{definition}\label{df2}
	The fractional space-time Fourier transform of $f \in L^1(\mathbb{R}^{3,1},C\ell_{3,1})$ with respect to a parameter $\alpha$, is defined as \cite{bib1}
	\begin{equation}
		\mathcal{F}_{rSFT}^\alpha[f](\boldsymbol{w})=\int_{\mathbb{R}^{3,1}}\boldsymbol{e}^{-e_tw_tt}\left.f(\boldsymbol{x})\mathcal{K}_{i_3}^\alpha(\underline{x},\underline{w})\right.d^4\boldsymbol{x},\label{eq23}
	\end{equation}
	the right kernel factor $\mathcal{K}_{i_3}^\alpha$ is
	\begin{equation}
		\mathcal{K}_{i_3}^{\alpha}(\underline{x},\underline{w})=\begin{cases}(\csc\alpha)^{\frac32}\boldsymbol{e}^{i_3\left((\underline{x}^2+\underline{w}^2)\frac{\cot\alpha}2-\underline{x}\cdot\underline{w}\csc\alpha+\frac{2\alpha-\pi}4\right)},&\alpha\neq n\pi\\\delta(\underline{w}-\underline{x}),&\alpha=2n\pi\\\delta(\underline{w}+\underline{x}),&\alpha=(2n+1)\pi\end{cases},\label{eq24}
	\end{equation}
	where $sin\alpha>0$. It is assumed that $\alpha \neq n\pi$.
\end{definition}

The inversion formula for the FrSFT is given by
\begin{equation}
	f(\boldsymbol{x})=\frac{(\csc\alpha)^{\frac32}}{(2\pi)^4}\int_{\mathbb{R}^{3,1}}\boldsymbol{e}^{e_tw_tt}\mathcal{F}_{rSFT}^\alpha[f](\boldsymbol{w})\mathcal{K}_{-i_3}^\alpha(\underline{x},\underline{w})d^4\boldsymbol{w},\label{eq25}
\end{equation}
where the space-time frequency volume $d^4\boldsymbol{w}=dw_tdw_1dw_2dw_3dw_4$.

Based on the existing research results, we introduce the concept of linear canonical space-time transform (LCST) by replacing the fractional Fourier kernel in the fractional space-time Fourier transform with a linear canonical kernel, and propose the fundamental properties of the proposed linear canonical space-time transform.

\section{Linear canonical space-time transform}
\label{sec3}
\subsection{Definition}
\label{subsec3}

\begin{definition}\label{df3}
	The linear canonical space-time transform of $f \in L^1(\mathbb{R}^{3,1},C\ell_{3,1})$ with respect to a parameter matrix $A=(a,b;c,d) \in \mathbb{R}^{2\times2}$, is defined as
	\begin{equation}
		\mathcal{L}_{A}[f](\boldsymbol{w})=\int_{\mathbb{R}^{3,1}}\boldsymbol{e}^{-e_tw_tt}\left.f(\boldsymbol{x})\mathcal{K}_{i_3}^A(\underline{x},\underline{w})\right.d^4\boldsymbol{x},\label{eq26}
	\end{equation}
	here we modify the right kernel factor $\mathcal{K}_{i_3}^A$ to
	\begin{align*}
	\mathcal{K}_{i_3}^{A}(\underline{x},\underline{w})=\begin{cases}\frac{1}{(2\pi b)^{\frac{3}{2}}} \boldsymbol{e}^{i_3\left(\frac{a}{2b}\underline{x}^2+\frac{d}{2b}\underline{w}^2-\frac{\underline{x}\cdot\underline{w}}{b}\right)},&b\neq 0\\\sqrt{d} \boldsymbol{e}^{i_3\left(\frac{cd}{2}\underline{w}^2\right)}\delta(\underline{x}-d\underline{w}),&b=0\\\end{cases}.\label{eq27}
	\end{align*}
	where $a,b,c,d \in \mathbb{R}$, and $ad-bc=1$.
\end{definition}

\subsection{Properties}
\label{subsec4}

\begin{theorem}\label{thm3}
(Inversion of LCST) Let $f,~\mathcal{L}_{A}[f]\in L^1(\mathbb{R}^{3,1},C\ell_{3,1})$, then $f$ can be recovered from $~\mathcal{L}_{A}[f]$ by using the formula
	\begin{equation}
		f(\boldsymbol{x})=\frac{1}{2\pi} \int_{\mathbb{R}^{3,1}}\boldsymbol{e}^{e_tw_tt}\mathcal{L}_{A}[f](\boldsymbol{w})\mathcal{K}_{-i_3}^A(\underline{x},\underline{w})d^4\boldsymbol{w},\label{eq28}
	\end{equation}
	where the space-time frequency volume $d^4\boldsymbol{w}=dw_tdw_1dw_2dw_3dw_4$.
\end{theorem}

\begin{proof}
	Based on identity (\ref{eq16}) and Definition \ref{df3}, we obtain
	\begin{equation}\nonumber
		\begin{aligned}
			&\frac{1}{2\pi}\int_{\mathbb{R}^{3,1}}\boldsymbol{e}^{e_{t}w_{t}t}\mathcal{L}_{A}[f](\boldsymbol{w})\mathcal{K}_{-i_{3}}^{A}(\underline{x},\underline{w})d^4\boldsymbol{w} \\
			&=\frac{1}{2\pi}\int_{\mathbb{R}^{3,1}}\int_{\mathbb{R}^{3,1}}\boldsymbol{e}^{e_tw_tt}\boldsymbol{e}^{-e_tw_ts}f(\boldsymbol{y})\mathcal{K}_{i_3}^{A}(\underline{y},\underline{w})\mathcal{K}_{-i_3}^A(\underline{x},\underline{w})d^4\boldsymbol{y}d^4\boldsymbol{w} \\
			&=\frac{1}{2\pi}\int_{\mathbb{R}^{3,1}}\left[\int_{\mathbb{R}}\boldsymbol{e}^{w_t(t-s)e_t}dw_t\right]f(y)\boldsymbol{e}^{\frac{-i_3a}{2b}\left(\underline{x}^2-\underline{y}^2\right)} \\
			&\times\left[\int_{\mathbb{R}^3}\boldsymbol{e}^{\frac{i_3}{b}\left((\underline{x}-\underline{y})\cdot\underline{w}\right)}dw_1dw_2dw_3\right]d^4\boldsymbol{y} \\
			&=\frac{1}{(2\pi)^4b^3}\int_{\mathbb{R}^{3,1}}2\pi\delta(t-s)\left.f(\boldsymbol{y})\right.e^{\frac{-i_3a}{2b}\left(\underline{x}^2-\underline{y}^2\right)}\left(2\pi b\right)^3\delta(\underline{x}-\underline{y})d^4\boldsymbol{y} \\
			&=f(\boldsymbol{x}).
		\end{aligned}
	\end{equation}
\end{proof}

\begin{theorem}\label{thm4}
	(Properties) Let $f,g \in L^1(\mathbb{R}^{3,1},C\ell_{3,1})$, with $\boldsymbol{x}, \boldsymbol{w} \in \mathbb{R}^{3,1}$, constants $\boldsymbol{u}=(u_t,\underline{u}), \boldsymbol{y}=(y_t,\underline{y}) \in \mathbb{R}^{3,1}$. We have
	
	\noindent
	(1). Right-hand linearity that commute with $i_3$ and left-hand linearity that commute with $e_t$:
	\begin{equation}\nonumber
		\mathcal{L}_{A}[Mf+gN](\boldsymbol{w})=M\mathcal{L}_{A}[f](\boldsymbol{w})+\mathcal{L}_{A}[g](\boldsymbol{w})N,
	\end{equation}
	with$M=M_0+M_te_t+M_{12}e_{12}+M_{23}e_{23}+M_{13}e_{13}+M_{t12}e_{t12}+M_{t23}e_{t23}+M_{t13}e_{t13}$, $N=N_0+N_1e_1+N_2e_2+N_3e_3+N_{12}e_{12}+N_{23}e_{23}+N_{13}e_{13}+N_{123}i_3$.
	
	\noindent
	(2). Modulation covariance:
	\begin{equation}\nonumber
		\begin{aligned}
			\mathcal{L}_{A}{\left[\boldsymbol{e}^{-e_tu_tt}\left.f(\boldsymbol{x})\right.\boldsymbol{e}^{-i_3\underline{x}\cdot\underline{u}}\right]}(\boldsymbol{w})&=\mathcal{L}_{A}{\left[f(\boldsymbol{x})\right]}\left(\boldsymbol{w}+(u_te_t+b\underline{u})\right)\\&\times \boldsymbol{e}^{\frac{-i_3}2d\left(b\underline{u}^2+2\underline{w}\cdot\underline{u}\right)}.
		\end{aligned}
	\end{equation}
	
	\noindent
	(3). Translation covariance:
	\begin{equation}\nonumber
		\begin{gathered}
			\mathcal{L}_{A}\bigl[f(\boldsymbol{x}-\boldsymbol{y})\bigr](\boldsymbol{w}) \left.=\boldsymbol{e}^{-e_ty_tw_t}\mathcal{L}_{A}[f]\left(w_te_t+\underline{w}-a\right.\underline{y}\right) \\
			\times \boldsymbol{e}^{i_3\frac{1-a}{b}\left(\frac{a\underline{y}^2}2-\underline{y}\cdot\underline{w}\right)}. 
		\end{gathered}
	\end{equation}
	
	\noindent
	(4). Reflection: \quad $\mathcal{L}_{A}[f(-\boldsymbol{x})](\boldsymbol{w})=\mathcal{L}_{A}[f(\boldsymbol{x})](-\boldsymbol{w})$.
\end{theorem}

\begin{proof}
We only give the proof of (2) as an example.
	\begin{equation}\nonumber
		\begin{aligned}
			&\mathcal{L}_{A}{\left[f(\boldsymbol{x})\right]}\left(\boldsymbol{w}+(u_te_t+b\underline{u})\right) \times \boldsymbol{e}^{\frac{-i_3}2d\left(b\underline{u}^2+2\underline{w}\cdot\underline{u}\right)}\\
			=&\int_{\mathbb{R}^{3,1}}\boldsymbol{e}^{-e_t(w_t+u_t)t}\left.f(\boldsymbol{x})\mathcal{K}_{i_3}^A(\underline{x},\underline{w}+b\underline{u})\right.d^4\boldsymbol{x}\times \boldsymbol{e}^{\frac{-i_3}2d\left(b\underline{u}^2+2\underline{w}\cdot\underline{u}\right)}\\
			=&\int_{\mathbb{R}^{3,1}}\boldsymbol{e}^{-e_t(w_t+u_t)t}\left.f(\boldsymbol{x})\frac{1}{(2\pi b)^{\frac{3}{2}}} \boldsymbol{e}^{i_3\left(\frac{a}{2b}\underline{x}^2+\frac{d}{2b}(\underline{w}+b\underline{u})^2-\frac{\underline{x}\cdot(\underline{w}+b\underline{u})}{b}\right)}\right.d^4\boldsymbol{x}\times \boldsymbol{e}^{\frac{-i_3}2d\left(b\underline{u}^2+2\underline{w}\cdot\underline{u}\right)}\\
			=&\int_{\mathbb{R}^{3,1}}\boldsymbol{e}^{-e_tw_tt}[\boldsymbol{e}^{-e_tu_tt}\left.f(\boldsymbol{x})\boldsymbol{e}^{-i_3\underline{x}\cdot\underline{u}}]\mathcal{K}_{i_3}^A(\underline{x},\underline{w})\right.d^4\boldsymbol{x}\\
			=&\mathcal{L}_{A}{\left[\boldsymbol{e}^{-e_tu_tt}\left.f(\boldsymbol{x})\right.\boldsymbol{e}^{-i_3\underline{x}\cdot\underline{u}}\right]}(\boldsymbol{w}).
		\end{aligned}
	\end{equation}
\end{proof}

\begin{theorem}\label{thm5}
	(Plancherel theorem)
	Let $\mathcal{L}_{A}[f_1]$ and $\mathcal{L}_{A}[f_2]$ be the LCSTs of $ f_1$ and $f_2 \in L^{2}(\mathbb{R}^{3,1},C\ell_{3,1})$, respectively. Then, we have
	\begin{equation}
		\left\langle\mathcal{L}_{A}[f_1],\mathcal{L}_{A}[f_2]\right\rangle_{L^2(\mathbb{R}^{3,1},C\ell_{3,1})}=2 \pi\langle f_1,f_2\rangle_{L^2(\mathbb{R}^{3,1},C\ell_{3,1})}.\label{eq30}
	\end{equation}
\end{theorem}

\begin{proof}
	Using the definition of \ref{df3} and Fubini replacement theorem, we get the following
	\begin{align*}\nonumber
	&\big\langle\mathcal{L}_{A}\big[f_{1}\big],\mathcal{L}_{A}\big[f_{2}\big]\big\rangle_{L^{2}(\mathbb{R}^{3,1},C\ell_{3,1})} \\
			&=\frac{1}{(2\pi b)^3}\int_{\mathbb{R}^{3,1}}\int_{\mathbb{R}^{3,1}}\int_{\mathbb{R}}Tr\bigg[\boldsymbol{e}^{-e_{t}w_{t}t}f_{1}(\mathbf{x})\exp\left\{\frac{i_{3}}{2}\left(\underline{x}^{2}-\underline{y}^{2}\right)\frac{a}{b} \cot\alpha\right\} \\
			&\left.\times\int_{\mathbb{R}^{3}}\exp\left\{-i_{3}\left((\underline{x}-\underline{y})\cdot\underline{w}\right)\frac{1}{b} \right\}dw_{1}dw_{2}dw_{3}\widetilde{f_{2}(\mathbf{y})}\exp\left\{e_{t}w_{t}s\right\}\right]dw_{t}d^{4}\mathbf{x}d^{4}\mathbf{y} \\
			&\stackrel{(\ref{eq16})}{=}\frac{1}{(2\pi b)^3}\int_{\mathbb{R}^{3,1}}\int_{\mathbb{R}^{3,1}}\int_{\mathbb{R}}Tr\bigg[\boldsymbol{e}^{-e_{t}w_{t}t}f_{1}(\mathbf{x})\exp\left\{\frac{i_{3}}{2}\left(\underline{x}^{2}-\underline{y}^{2}\right)\frac{a}{b} \right\} \\
			&\left.\times\left(2\pi b\right)^{3}\delta(\underline{x}-\underline{y})\widetilde{f_{2}(\mathbf{y})}\exp\left\{e_{t}w_{t}s\right\}\right]dw_{t}d^{4}\mathbf{x}d^{4}\mathbf{y} \\
			&\stackrel{(\ref{eq12})}{=} \int_{\mathbb{R}^{3,1}}\int_{\mathbb{R}^{3,1}}\int_{\mathbb{R}}Tr\bigg[\exp{\{-e_{t}w_{t}(t-s)\}}f_{1}(\mathbf{x}) \\
			&\left.\times\exp\left\{\frac{i_{3}}{2}\left(\underline{x}^{2}-\underline{y}^{2}\right)\frac{a}{b} \right\}\delta(\underline{x}-\underline{y})\widetilde{f_{2}(\mathbf{y})}\right]dw_{t}d^{4}\mathbf{x}d^{4}\mathbf{y} \\
			&=\int_{\mathbb{R}^{3,1}}\int_{\mathbb{R}^{3,1}}Tr\bigg[\int_{\mathbb{R}}\exp\left\{-e_{t}w_{t}(t-s)\right\}dw_{t}f_{1}(\mathbf{x}) \\
			&\left.\times\exp\left\{\frac{i_{3}}{2}\left(\underline{x}^{2}-\underline{y}^{2}\right)\frac{a}{b} \right\}\delta(\underline{x}-\underline{y})\widetilde{f_{2}(\mathbf{y})}\right]d^{4}\mathbf{x}d^{4}\mathbf{y}\\
			&\stackrel{(\ref{eq16})}{=}\int_{\mathbb{R}^{3,1}}\int_{\mathbb{R}^{3,1}}Tr\Big[2\pi\delta(t-s)f_1(\mathbf{x})\exp\left\{\frac{i_3}2\left(\underline{x}^2-\underline{y}^2\right)\frac{a}{b} \cot\alpha\right\} \\
			&\times\left.\delta(\underline{x}-\underline{y})\widetilde{f_2(\mathbf{y})}\right]d^4\mathbf{x}d^4\mathbf{y} \\
			&=2 \pi \int_{\mathbb{R}^{3,1}}Tr\left[f_1(\mathbf{x})\widetilde{f_2(\mathbf{x})}\right]d^4\mathbf{x} \\
			&=2 \pi\left<f_1,f_2\right>_{L^2(\mathbb{R}^{3,1},C\ell_{3,1})}.
	\end{align*}
\end{proof}

\begin{theorem}\label{thm6}
	(Parseval formula) If $f_1=f_2=f$, then Parseval's identity
	\begin{equation}
		\left\|\mathcal{L}_{A}[f]\right\|_2=\sqrt{2 \pi}\left\|f\right\|_2.\label{eq31}
	\end{equation}
	holds.
\end{theorem}

\begin{theorem}\label{thm8}
	(Partial derivatives) For $k=1,2,3$, if $\begin{aligned}f,\frac{\partial f}{\partial t},\frac{\partial f}{\partial x_k},\in L^1(\mathbb{R}^{3,1},C\ell_{3,1})\end{aligned}$, provided that the derivatives exist, then
	
	\noindent
	(1). $\mathcal{L}_{A}\left[\frac{\partial f(\boldsymbol{x})}{\partial t}\right](\boldsymbol{w})=e_tw_t\mathcal{L}_{A}[f](\boldsymbol{w})$.
	
	\noindent
	(2). $\mathcal{L}_{A}\left[\frac{\partial f(\boldsymbol{x})}{\partial x_k}\right](\boldsymbol{w})=\left[w_k\mathcal{L}_{A}[f](\boldsymbol{w})~-~a~\mathcal{L}_{A}[x_k~f(\boldsymbol{x})](\boldsymbol{w})\right]\times \frac{i_3}{b}$.
	
	\noindent
	(3). $\frac\partial{\partial w_{t}}\Bigl\{\mathcal{L}_{A}\left[f\right](\boldsymbol{w})\Bigr\}=-e_{t}\mathcal{L}_{A}[t\left.f(\boldsymbol{x})\right](\boldsymbol{w})$.
	
	\noindent
	(4). $\frac{\partial}{\partial w_{k}}\Big\{\mathcal{L}_{A}[f](\boldsymbol{w})\Big\}=\Big[a\mathcal{L}_{A}\Big[w_{k}f(\boldsymbol{x})\Big](\boldsymbol{w})-\mathcal{L}_{A}\Big[x_{k}f(\boldsymbol{x})\Big](\boldsymbol{w})\Big]\times\frac{i_3}{b}$.
\end{theorem}
\begin{proof}
	We will just prove (2). For $k = 1, 2, 3$, we have
	\begin{align*}\nonumber
	&\mathcal{L}_{A}\left[\frac{\partial f(x)}{\partial x_k}\right](\boldsymbol{w}) \\
			&=\int_{\mathbb{R}^{3,1}}\boldsymbol{e}^{-e_tw_tt}\frac{\partial f(\boldsymbol{x})}{\partial x_k}\mathcal{K}_{i_3}^{A}(\underline{x},\underline{w})d^4\boldsymbol{x}\\
			&=\int_{\mathbb{R}^3}\boldsymbol{e}^{-e_tw_tt}\left(\int_{\mathbb{R}}\frac{\partial}{\partial x_k}f(\boldsymbol{x})\mathcal{K}_{i_3}^A(\underline{x},\underline{w})dx_k\right)d^3\boldsymbol{x} \\
			&=\frac{1}{(2\pi b)^{\frac{3}{2}}}\int_{\mathbb{R}^3}\boldsymbol{e}^{-e_tw_tt} \\
			&\times\left(\int_{\mathbb{R}}\frac\partial{\partial x_k}\left.f(\boldsymbol{x})\right.\boldsymbol{e}^{i_3(\frac{a}{2b}\underline{x}^2+ \frac{d}{2b}\underline{w}^2- \frac{1}{b}\underline{x}\cdot \underline{w})}dx_k\right)d^3\boldsymbol{x} \\
			&=\frac{1}{(2\pi b)^{\frac{3}{2}}}\int_{\mathbb{R}^3}\boldsymbol{e}^{-e_tw_tt} \\
			&\times\left(\int_{\mathbb{R}}f\left(\boldsymbol{x}\right)\boldsymbol{e}^{i_3\left(\frac{a}{2b}\underline{x}^2+\frac{d}{2b}\underline{w}^2-\frac{1}{b}\underline{x}\cdot\underline{w}\right)}\right|_{x_k=-\infty}^{x_k=\infty} \\
			&\left.-\left.f(\boldsymbol{x})\frac{(ax_k-w_k)i_3}{b}\right.\boldsymbol{e}^{i_3\left(\frac{a}{2b}\underline{x}^2+\frac{d}{2b}\underline{w}^2-\frac{1}{b}\underline{x}\cdot\underline{w}\right)}dx_k\right)d^3\boldsymbol{x} \\
			&=-\frac{1}{(2\pi b)^{\frac{3}{2}}}\int_{\mathbb{R}^3}\boldsymbol{e}^{-e_tw_tt} \\
			&\times\left(\int_{\mathbb{R}}f\left(\boldsymbol{x}\right)\frac{\left(ax_k-w_k\right)i_3}{b}\boldsymbol{e}^{i_3\left(\frac{a}{2b}\underline{x}^2+\frac{d}{2b}\underline{w}^2-\frac{1}{b}\underline{x}\cdot\underline{w}\right)}dx_k\right)d^3\boldsymbol{x}\\
			&=\frac{1}{(2\pi b)^{\frac{3}{2}}}\int_{\mathbb{R}^3}\boldsymbol{e}^{-e_tw_tt} \\
			&\times\left(\int_{\mathbb{R}}f(\boldsymbol{x})(w_k-ax_k)\boldsymbol{e}^{i_3\left(\frac{a}{2b}\underline{x}^2+\frac{d}{2b}\underline{w}^2-\frac{1}{b}\underline{x}\cdot\underline{w}\right)}dx_k\right)d^3\boldsymbol{x}\frac{i_3}{b}  \\
			&=\int_{\mathbb{R}^3}\boldsymbol{e}^{-e_tw_tt}\left(\int_{\mathbb{R}}f(\boldsymbol{x})(w_k-x_k\cos\alpha)\mathcal{K}_{i_3}^\alpha(\underline{x},\underline{w})dx_k\right)d^3\boldsymbol{x}i_3\frac{i_3}{b}  \\
			&=\int_{\mathbb{R}^{3,1}}\boldsymbol{e}^{-\boldsymbol{e}_{t}w_{t}t}f(\boldsymbol{x})(w_{k}-ax_{k})\mathcal{K}_{i_{3}}^{A}(\underline{x},\underline{w})d^{4}\boldsymbol{x}\frac{i_3}{b}  \\
			&=\left[w_{k}\mathcal{L}_{A}\bigl[f\bigr](\boldsymbol{w})-a\mathcal{L}_{A}\bigl[x_{k}f(\boldsymbol{x})\bigr](\boldsymbol{w})\right]\frac{i_3}{b}.
	\end{align*}
\end{proof}

\section{Convolution theorems for the LCST}
\label{sec4}

Convolution stands as a fundamental concept in linear time-invariant systems. It characterizes the output of a continuous-time linear time-invariant system as the convolution of the input signal and the system's impulse response. Essentially, convolution entails an integral operation and finds extensive application across fields like numerical analysis, finite impulse response, pattern recognition, and signal processing. Consequently, delving deeper into the convolution operation's form and the associated convolution theorem within the domain of LCT under space-time algebra bears significant theoretical and practical relevance.

We define the convolution of two space-time signals $a,b \in L^{1}(\mathbb{R}^{3,1},C\ell_{3,1})$ as
\begin{equation}
(a\star b)(\boldsymbol{x})=\int_{\mathbb{R}^{3,1}}a(\boldsymbol{y})b(\boldsymbol{x}-\boldsymbol{y})d^4\boldsymbol{y},\label{eq4.1}
\end{equation}
provided that the integral exists.

The Mustard convolution of two space-time signals $a,b \in L^{1}(\mathbb{R}^{3,1},C\ell_{3,1})$ is defined as
\begin{equation}
(a\star_Mb)(\boldsymbol{x})=\mathcal{F}^{-1}(\mathcal{F}\{a\}\mathcal{F}\{b\}).\label{eq4.2}
\end{equation}
provided that the integral exists.

\begin{theorem}\label{thm4.1}
(Mustard convolution in terms of standard convolution)\cite{bib32}. The Mustard convolution (\ref{eq4.2}) of two space-time functions $a,b \in L^{1}(\mathbb{R}^{3,1},C\ell_{3,1})$ can be expressed in terms of eight standard convolutions (\ref{eq4.1}) as
\begin{equation}
\begin{aligned}
a\star_{M}b(\boldsymbol{x})& =[a_{+}\star b_{+}(\boldsymbol{x})]_{+}+[a_{+}\star b_{+}^{(1,1)}(t_{1},-\underline{x})]_{-} \\
&+[a_+\star b_-^{(1,0)}(\boldsymbol{x})]_++[a_+\star b_-^{(0,1)}(t_1,-\underline{x})]_- \\
&+[a_-\star b_+^{(0,1)}(t_1,-\underline{x})]_++[a_-\star b_+^{(1,0)}(\boldsymbol{x})]_- \\
&+[a_-\star b_-^{(1,1)}(t_1,-\underline{x})]_++[a_-\star b_-(\boldsymbol{x})]_-.
\end{aligned}\label{eq4.3}
\end{equation}
\end{theorem}

The computation of a typical term in Theorem \ref{thm4.1} can be illustrated by writing out, eg, the third term in full detail
\begin{equation}
\begin{aligned}
[a_{+}\star b_{-}^{(1,0)}(\boldsymbol{x})]_{+}& =\frac{1}{2}\Bigg[\int_{\mathbb{R}^{3,1}}\frac{1}{2}(a(\boldsymbol{y})+e_{t}a(\boldsymbol{y})i_{3})  \\
&\frac12(b(-(t_x-t_y),\underline{x}-\underline{y})-e_tb(-(t_x-t_y),\underline{x}-\underline{y})i_3)d^4\boldsymbol{y} \\
&+e_t\int_{\mathbb{R}^{3,1}}\frac12(a(\boldsymbol{y})+e_ta(\boldsymbol{y})i_3) \\
&\frac12(b(-(t_x-t_y),\underline{x}-\underline{y})-e_tb(-(t_x-t_y),\underline{x}-\underline{y})i_3)d^4\boldsymbol{y} i_3\Bigg].
\end{aligned}\label{eq4.4}
\end{equation}

We define a new convolution structure for LCST based on the relationship between the LCST and the SFT, which states that a modified ordinary convolution in the time domain is equivalent to a simple multiplication operation for LCST and SFT. Based on the definition of the ordinary SFT (\ref{eq21}), the relationship between the LCST and the SFT is given below:

\begin{lemma}\label{lem4.1}
Let $f \in L^{1}(\mathbb{R}^{3,1},C\ell_{3,1})$ and $\alpha \neq n\pi$, then
\begin{equation}
\mathcal{L}_{A}\big[f\big](\boldsymbol{w})=B\mathcal{F}_{SFT}[f(\boldsymbol{x})\boldsymbol{e}^{i_3\frac{a}{2b}\underline{x}^2}](w_{t},\frac{\underline{w}}{b})\boldsymbol{e}^{\frac{i_3d}{2b} \underline{w}^2},\label{eq4.5}
\end{equation}
and
\begin{equation}
\left|B\mathcal{L}_{A}\bigl[f\bigr](w_{t},\underline{w}b)\right|=\left|\mathcal{F}_{SFT}\biggl[f(\boldsymbol{x})\boldsymbol{e}^{i_{3}\frac{a}{2b} \underline{x}^{2}}\biggr](\boldsymbol{w})\biggr|.\right. \label{eq4.6}
\end{equation}
where $B=\frac{1}{(2\pi b)^{3/2}}$.
\end{lemma}

\begin{definition}\label{df4.1}
For any space-time function $f(\boldsymbol{x}) \in L^{1}(\mathbb{R}^{3,1},C\ell_{3,1})$, let us define the function $\tilde{f}(\boldsymbol{x})=f(\boldsymbol{x})e^{i_3\frac{a}{2b}\underline{x}^2}$. For any two functions $f$ and $g$, we define the convolution operation $\odot$ by
\begin{equation}
h(\boldsymbol{x})=(f\odot g)(\boldsymbol{x})=e^{i_3\frac{a}{2b}\underline{x}^2}\left(\tilde{f} \star_M g\right)(\boldsymbol{x}).\label{eq4.7}
\end{equation}
where $\star_M$ is the Mustard convolution operation for the SFT as defined by (\ref{eq4.2}).
\end{definition}

Now we state and prove our new convolution theorem.
\begin{theorem}\label{thm4.2}
Let $h(\boldsymbol{x})=(f\odot g)(\boldsymbol{x})$ and $\mathcal{L}_{A}[h](\boldsymbol{w})$, $\mathcal{L}_{A}[f](\boldsymbol{w})$ denote the LCST of $h(\boldsymbol{x})$ and $f(\boldsymbol{x})$. $\mathcal{F}_{SFT}[g](\boldsymbol{w})$ denotes the SFT of $g(\boldsymbol{x})$. Then
\begin{equation}
\mathcal{L}_{A}[h(\boldsymbol{x})](\boldsymbol{w})=\mathcal{L}_{A}[f(\boldsymbol{x})](\boldsymbol{w})\mathcal{F}_{SFT}[g(\boldsymbol{x})](w_t, \frac{\underline{w}}{b}).\label{eq4.8}
\end{equation}
\end{theorem}

\begin{proof}
From the definition of the LCST and Definition \ref{df4.1}, we have
\begin{equation}
\begin{aligned}
\mathcal{L}_{A}[h]&=B\int_{\mathbb{R}^{3,1}}\boldsymbol{e}^{-e_tw_tt}h(\boldsymbol{x})\boldsymbol{e}^{i_3(\frac{a}{2b}\underline{x}^2+\frac{d}{2b}\underline{w}^2-\frac{\underline{x}\cdot\underline{w}}{b})}d^4\boldsymbol{x}\\
&=B\mathcal{F}_{SFT}[h(\boldsymbol{x})\boldsymbol{e}^{i_3\frac{a}{2b}\underline{x}^2}](w_t, \frac{\underline{w}}{b})\boldsymbol{e}^{i_3\frac{d}{2b}\underline{w}^2}\\
&=B\mathcal{F}_{SFT}[(\tilde{f} \star_M g)(\boldsymbol{x})\boldsymbol{e}^{i_3\frac{a}{2b}\underline{x}^2}](w_t, \frac{\underline{w}}{b})\boldsymbol{e}^{i_3\frac{d}{2b}\underline{w}^2}.\\
\end{aligned}\label{eq4.9}
\end{equation}
According to convolution theorem for SFT and (\ref{eq4.5}), we have
\begin{equation}
\begin{aligned}
\mathcal{L}_{A}[h]&=B\boldsymbol{e}^{i_3\frac{d}{2b}\underline{w}^2}\mathcal{F}_{SFT}[f(\boldsymbol{x})\boldsymbol{e}^{i_3\frac{a}{2b}\underline{x}^2} \star_M g(\boldsymbol{x})](w_t, \frac{\underline{w}}{b})\\
&=B\boldsymbol{e}^{i_3\frac{d}{2b}\underline{w}^2}\mathcal{F}_{SFT}[f(\boldsymbol{x})\boldsymbol{e}^{i_3\frac{a}{2b}\underline{x}^2}](w_t, \frac{\underline{w}}{b})\mathcal{F}_{SFT}[g(\boldsymbol{x})](w_t, \frac{\underline{w}}{b})\\
&=\mathcal{L}_{A}[f(\boldsymbol{x})]\mathcal{F}_{SFT}[g(\boldsymbol{x})](w_t, \frac{\underline{w}}{b}),\\
\end{aligned}\label{eq4.10}
\end{equation}
and this completes the proof of (\ref{eq4.8}), that is
\begin{equation}
\mathcal{L}_{A}[f \star_M g(\boldsymbol{x})](\boldsymbol{w})=\mathcal{L}_{A}[f(\boldsymbol{x})](\boldsymbol{w})\mathcal{F}_{SFT}[g(\boldsymbol{x})](w_t, \frac{\underline{w}}{b}).\label{eq4.11}
\end{equation}
\end{proof}

\begin{definition}\label{df4.2}
For any space-time function $f(\boldsymbol{x}) \in L^{1}(\mathbb{R}^{3,1},C\ell_{3,1})$, let us define the function $\hat{f}(\boldsymbol{x})=f(\boldsymbol{x})e^{-i_3\frac{d}{2b}\underline{w}^2}$. For any two functions $f$ and $g$, we define the convolution operation $\otimes$ by
\begin{equation}
h(\boldsymbol{x})=(f\otimes g)(\boldsymbol{x})=e^{i_3\frac{d}{2b}\underline{x}^2}\left(\hat{f} \star g\right)(\boldsymbol{x}).\label{eq4.12}
\end{equation}
where $\star$ is the conventional convolution operation for the SFT as defined by (\ref{eq4.1}).
\end{definition}

Now we state and prove our new product theorem.
\begin{theorem}\label{thm4.3}
Let $h(\boldsymbol{x})=(f\otimes g)(\boldsymbol{x})$ and $\mathcal{L}_{A}[h](\boldsymbol{w})$, $\mathcal{L}_{A}[f](\boldsymbol{w})$ denote the LCST of $h(\boldsymbol{x})$ and $f(\boldsymbol{x})$. $\mathcal{F}_{SFT}[g](\boldsymbol{w})$ denotes the SFT of $g(\boldsymbol{x})$. Then
\begin{equation}
\mathcal{L}_{A}[f(\boldsymbol{x})g(\frac{\boldsymbol{x}}{b})]=\frac{1}{(2\pi)^4}(\mathcal{L}_{A}[f] \otimes \mathcal{F}_{SFT}[g]).\label{eq4.13}
\end{equation}
\end{theorem}

\begin{proof}
According to Definition \ref{df4.2}, we have
\begin{equation}
\begin{aligned}
\mathcal{L}_{A}[f] \otimes \mathcal{F}_{SFT}[g]&=\boldsymbol{e}^{i_3\frac{d}{2b}\underline{w}^2}\hat{\mathcal{L}_{A}}[f] \star \mathcal{F}_{SFT}[g]\\
&=\boldsymbol{e}^{i_3\frac{d}{2b}\underline{w}^2}\mathcal{L}_{A}[f(\boldsymbol{y})]\boldsymbol{e}^{-i_3\frac{d}{2b}\underline{w}^2}\mathcal{F}_{SFT}[g(\boldsymbol{x}-\boldsymbol{y})]d^{4}\boldsymbol{y}.\\
\end{aligned}\label{eq4.14}
\end{equation}
Making use of the definition of the LCST, we obtain
\begin{equation}
\begin{aligned}
\mathcal{L}_{A}[f] \otimes \mathcal{F}_{SFT}[g]&=\boldsymbol{e}^{i_3\frac{d}{2b}\underline{w}^2}B\int_{\mathbb{R}^{3,1}}\boldsymbol{e}^{-e_tw_tt}f(\boldsymbol{x})\boldsymbol{e}^{i_3(\frac{a}{2b}\underline{x}^2-\frac{\underline{x}\cdot\underline{w}}{b})}d^4\boldsymbol{x}\\
&\int_{\mathbb{R}^{3,1}}\mathcal{F}_{SFT}[g(\boldsymbol{x}-\boldsymbol{y})]d^{4}\boldsymbol{y}\\
&=\boldsymbol{e}^{i_3\frac{d}{2b}\underline{w}^2}B\int_{\mathbb{R}^{3,1}}\boldsymbol{e}^{-e_tw_tt}f(\boldsymbol{x})\boldsymbol{e}^{i_3(\frac{a}{2b}\underline{x}^2)}\boldsymbol{e}^{-i_3\frac{\underline{x}\cdot (\underline{w}-\boldsymbol{u})}{b}}d^4\boldsymbol{x}\\
&\int_{\mathbb{R}^{3,1}}\mathcal{F}_{SFT}[g(\boldsymbol{u})]d^{4}\boldsymbol{u}\\
&=(2\pi)^4\int_{\mathbb{R}^{3,1}}\boldsymbol{e}^{-e_tw_tt}f(\boldsymbol{x})g(\frac{\boldsymbol{x}}{b})\boldsymbol{e}^{i_3\frac{a}{2b}\underline{x}^2}\boldsymbol{e}^{-i_3\frac{\underline{x}\cdot \underline{w}}{b}}\boldsymbol{e}^{i_3\frac{d}{2b}\underline{w}^2}d^{4}\boldsymbol{x}\\
&=(2\pi)^4\mathcal{L}_{A}[f(\boldsymbol{x})g(\frac{\boldsymbol{x}}{b})].\\
\end{aligned}\label{eq4.15}
\end{equation}
\end{proof}

\section{Convolution theorem for the two-sided LCST}
\label{sec5}

\begin{definition}\label{df5.1}
For any real matrix $M_k=\begin{pmatrix}
  A_k&B_k\\
  C_k&D_k
\end{pmatrix}$
with $A_kD_k-B_kC_k=1$, $k=1,2$. The two-sided linear canonical space-time transform maps 16-dimensional space-time algebra functions $h \in L^{1}(\mathbb{R}^{3,1},C\ell_{3,1})$ to 16-dimensional spectrum functions $\mathcal{L}_{A}[h]: \mathbb{R}^{3,1} \to C\ell_{3,1}$. It is defined as
\begin{equation}
\mathcal{L}_{A}[h](\boldsymbol{w})=\int_{\mathbb{R}^{3,1}}\mathcal{K}_{e_t}^{M_1}(w_t,t)h(\boldsymbol{x})\mathcal{K}_{i_3}^{M_2}(\underline{x},\underline{w})d^4\boldsymbol{x},\label{eq5.1}
\end{equation}
We give the left and the right kernel factors of the two-sided LCST as follows
\begin{equation}
\begin{aligned}
&\mathcal{K}_{e_t}^{M_{1}}(w_{t},t)=\frac{1}{\sqrt{2\pi B_{1}}}\boldsymbol{e}^{\frac{e_t}{2B_{1}}\left(A_{1}t^{2}+D_{1}w_{t}^{2}-2tw_{t}\right)}, \\
&\mathcal{K}_{i_3}^{M_{2}}(\underline{x},\underline{w})=\frac{1}{\sqrt{2\pi B_{2}}}\boldsymbol{e}^{\frac{i_3}{2B_{2}}\left(A_{2}\underline{x}^{2}+D_{2}\underline{w}^{2}-2\underline{x}\cdot \underline{w}\right)}.
\end{aligned}\label{eq5.2}
\end{equation}
\end{definition}

The inverse transformation of the proposed two-sided linear canonical space-time transform is given by the following theorem.
\begin{theorem}\label{thm5.1}
Let $\mathcal{L}_{A}[h](\boldsymbol{w})$ be the two-sided linear canonical space-time transform of any space-time valued function $h \in L^{2}(\mathbb{R}^{3,1},C\ell_{3,1})$. Then, $h(\boldsymbol{x})$ can be completely reconstructed via the formula
\begin{equation}
h(\boldsymbol{x})=\int_{\mathbb{R}^{3,1}}\mathcal{K}_{-e_t}^{M_1}(w_t,t)\mathcal{L}_{A}[h](\boldsymbol{w})\mathcal{K}_{-i_3}^{M_2}(\underline{x},\underline{w})d^4\boldsymbol{w}.\label{eq5.3}
\end{equation}
\end{theorem}

We define a new convolution of two space-time valued functions.
\begin{definition}\label{df5.2}
The new convolution of $f$ and $g$ is denoted by $f \star_N g$ and is defined as
\begin{equation}
\begin{aligned}
(f\star_Ng)(\boldsymbol{x})&=\int_{\mathbb{R}^{3,1}}\int_{\mathbb{R}^{3,1}}\int_{\mathbb{R}^{3,1}}\lambda_{e_t}(\boldsymbol{w})\mathcal{K}_{-e_t}^{M_{1}}(w_{t},t_1)\mathcal{K}_{e_t}^{M_{1}}(w_{t},t_1')f(\boldsymbol{y})\mathcal{K}_{i_3}^{M_{2}}(\underline{y},\underline{w}) \\
&\mathcal{K}_{e_t}^{M_{1}}(w_{t},t_1'')\mathcal{K}_{i_3}^{M_{2}}(\underline{z},\underline{w})\mathcal{K}_{-i_3}^{M_{2}}(\underline{x},\underline{w})d^4\boldsymbol{y}d^4\boldsymbol{z}d^4\boldsymbol{w},\\
\end{aligned}\label{eq5.4}
\end{equation}
where
\begin{equation}
\lambda_{e_t}(\boldsymbol{w})=(2\pi)^2B_1^{\frac{1}{2}}B_2^{\frac{7}{2}}\boldsymbol{e}^{e_t(\frac{D_1w_t^2}{2B_1}+\frac{D_2\underline{w}^2}{2B_2})}.\label{eq5.5}
\end{equation}
\end{definition}

We are now ready to present the main result of the section, which gives the expression of the proposed convolution (\ref{eq5.4}) in terms of the classical convolutions. In the proof of this result, we shall use the splitting technique. Moreover, for the shifting of kernels, we shall make the use of the following general identity:
\begin{equation}
\mathcal{K}_{e_t}^{M_i}h_\pm\mathcal{K}_{i_3}^{M_i}=\mathcal{K}_{e_t}^{M_i}\mathcal{K}_{\mp {e_t}}^{M_i}h_\pm=h_\pm \mathcal{K}_{i_3}^{M_i}\mathcal{K}_{\mp {e_t}}^{M_i}.\label{eq5.6}
\end{equation}

Furthermore, for any function $h : \mathbb{R}^{3,1} \to C\ell_{3,1}$ and a multi-index $(\phi_1, \phi_2)$ with $\phi_1, \phi_2 \in \left \{ 0,1 \right \} $, we set
\begin{equation}
h^{(\phi_1,\phi_2)}(\boldsymbol{x})=h\Big((-1)^{\phi_1}t_1,(-1)^{\phi_2}\underline{x}\Big).\label{eq5.7}
\end{equation}

\begin{theorem}\label{thm5.2}
The newly defined convolution (\ref{eq5.4}) of two space-time valued functions $f,g \in L^{2}(\mathbb{R}^{3,1},C\ell_{3,1})$ can be expressed in terms of the eight standard convolutions (\ref{eq4.1}) as
\begin{equation}
\begin{aligned}
f\star_{N}g(x)& =[f_{+}\star g_{+}(\boldsymbol{x})]_{+}+[f_{+}\star g_{+}^{(1,1)}(t_{1},-\underline{x})]_{-} \\
&+[f_+\star g_-^{(1,0)}(\boldsymbol{x})]_++[f_+\star g_-^{(0,1)}(t_1,-\underline{x})]_- \\
&+[f_-\star g_+^{(0,1)}(t_1,-\underline{x})]_++[f_-\star g_+^{(1,0)}(\boldsymbol{x})]_- \\
&+[f_-\star g_-^{(1,1)}(t_1,-\underline{x})]_++[f_-\star g_-(\boldsymbol{x})]_-.
\end{aligned}\label{eq5.8}
\end{equation}
\end{theorem}

\begin{proof}
Using the definition of two-sided linear canonical space-time transform, we have
\begin{align*}
f\star_{N}g(x)& =\mathcal{L}_{A}^{-1}(\lambda_{e_t}(\boldsymbol{w})\mathcal{L}_{A}[f(\boldsymbol{x})](\boldsymbol{w})\mathcal{L}_{A}[g(\boldsymbol{x})](\boldsymbol{w}))\\
&=\int_{\mathbb{R}^{3,1}}\mathcal{K}_{-e_t}^{M_1}(w_t,t_1)\lambda_{e_t}(\boldsymbol{w})\mathcal{L}_{A}[f(\boldsymbol{y})](\boldsymbol{w})\mathcal{L}_{A}[g(\boldsymbol{z})](\boldsymbol{w})\mathcal{K}_{-i_3}^{M_2}(\underline{x},\underline{w})d^4\boldsymbol{w}\\
&=\lambda_{e_t}(\boldsymbol{w})\int_{\mathbb{R}^{3,1}}\mathcal{K}_{-e_t}^{M_1}(w_t,t_1)\int_{\mathbb{R}^{3,1}}\mathcal{K}_{e_t}^{M_1}(w_t,t_1')f(\boldsymbol{y})\mathcal{K}_{i_3}^{M_2}(\underline{y},\underline{w})d^4\boldsymbol{y}\\
&\times \int_{\mathbb{R}^{3,1}}\mathcal{K}_{e_t}^{M_1}(w_t,t_1'')g(\boldsymbol{z})\mathcal{K}_{i_3}^{M_2}(\underline{z},\underline{w})d^4\boldsymbol{z}\mathcal{K}_{-i_3}^{M_2}(\underline{x},\underline{w})d^4\boldsymbol{w}\\
&=\int_{\mathbb{R}^{3,1}}\int_{\mathbb{R}^{3,1}}\int_{\mathbb{R}^{3,1}}\lambda_{e_t}(\boldsymbol{w})\mathcal{K}_{-e_t}^{M_1}(w_t,t_1)\mathcal{K}_{e_t}^{M_1}(w_t,t_1')f(\boldsymbol{y})\mathcal{K}_{i_3}^{M_2}(\underline{y},\underline{w})\\
&\times \mathcal{K}_{e_t}^{M_1}(w_t,t_1'')g(\boldsymbol{z})\mathcal{K}_{i_3}^{M_2}(\underline{z},\underline{w})\mathcal{K}_{-i_3}^{M_2}(\underline{x},\underline{w})d^4\boldsymbol{y}d^4\boldsymbol{z}d^4\boldsymbol{w}.\\
\end{align*}

Implementing the general $\pm$ splitting technique $f(\boldsymbol{y}) = f_{+}(\boldsymbol{y}) + f_{-}(\boldsymbol{y})$ and $g(\boldsymbol{z}) = g_{+}(\boldsymbol{z}) + g_{-}(\boldsymbol{z})$, we obtain
\begin{equation}
\begin{aligned}
f\star_{N}g(x)& =\int_{\mathbb{R}^{3,1}}\int_{\mathbb{R}^{3,1}}\int_{\mathbb{R}^{3,1}}\lambda_{e_t}(\boldsymbol{w})\mathcal{K}_{-e_t}^{M_1}(w_t,t_1)\mathcal{K}_{e_t}^{M_1}(w_t,t_1')\\
&\times [f_{+}(\boldsymbol{y})\mathcal{K}_{i_3}^{M_2}(\underline{y},\underline{w})\mathcal{K}_{e_t}^{M_1}(w_t,t_1'')g_{+}(\boldsymbol{z})\\
&+f_{+}(\boldsymbol{y})\mathcal{K}_{i_3}^{M_2}(\underline{y},\underline{w})\mathcal{K}_{e_t}^{M_1}(w_t,t_1'')g_{-}(\boldsymbol{z})\\
&+f_{-}(\boldsymbol{y})\mathcal{K}_{i_3}^{M_2}(\underline{y},\underline{w})\mathcal{K}_{e_t}^{M_1}(w_t,t_1'')g_{+}(\boldsymbol{z})\\
&+f_{-}(\boldsymbol{y})\mathcal{K}_{i_3}^{M_2}(\underline{y},\underline{w})\mathcal{K}_{e_t}^{M_1}(w_t,t_1'')g_{-}(\boldsymbol{z})]\\
&\times \mathcal{K}_{i_3}^{M_2}(\underline{z},\underline{w})\mathcal{K}_{-i_3}^{M_2}(\underline{x},\underline{w})d^4\boldsymbol{y}d^4\boldsymbol{z}d^4\boldsymbol{w}.\\
\end{aligned}\label{eq5.9}
\end{equation}

Applying the identities (\ref{eq5.6}) to shift inner kernel factors towards the left and right, we get
\begin{equation}
\begin{aligned}
f\star_{N}g(x)& =\int_{\mathbb{R}^{3,1}}\int_{\mathbb{R}^{3,1}}\int_{\mathbb{R}^{3,1}}\lambda_{e_t}(\boldsymbol{w})\mathcal{K}_{-e_t}^{M_1}(w_t,t_1)\mathcal{K}_{e_t}^{M_1}(w_t,t_1')\\
&\times [\mathcal{K}_{-e_t}^{M_2}(\underline{y},\underline{w})(f_{+}(\boldsymbol{y})g_{+}(\boldsymbol{z}))\mathcal{K}_{-i_3}^{M_1}(w_t,t_1'')\\
&+\mathcal{K}_{-e_t}^{M_2}(\underline{y},\underline{w})(f_{+}(\boldsymbol{y})g_{-}(\boldsymbol{z}))\mathcal{K}_{i_3}^{M_1}(w_t,t_1'')\\
&+\mathcal{K}_{e_t}^{M_2}(\underline{y},\underline{w})(f_{-}(\boldsymbol{y})g_{+}(\boldsymbol{z}))\mathcal{K}_{-i_3}^{M_1}(w_t,t_1'')\\
&+\mathcal{K}_{e_t}^{M_2}(\underline{y},\underline{w})(f_{-}(\boldsymbol{y})g_{-}(\boldsymbol{z}))\mathcal{K}_{i_3}^{M_1}(w_t,t_1'')]\\
&\times \mathcal{K}_{i_3}^{M_2}(\underline{z},\underline{w})\mathcal{K}_{-i_3}^{M_2}(\underline{x},\underline{w})d^4\boldsymbol{y}d^4\boldsymbol{z}d^4\boldsymbol{w}.\\
\end{aligned}\label{eq5.10}
\end{equation}

Again utilizing the general $\pm$ split and the notation similar to $f_{+}(\boldsymbol{y})g_{-}(\boldsymbol{z})=[f_{+}(\boldsymbol{y})g_{-}(\boldsymbol{z})]_{+}+[f_{+}(\boldsymbol{y})g_{-}(\boldsymbol{z})]_{-}$ for the second time split, and then applying the identities (\ref{eq5.6}) for shifting the right-sided kernels towards the left. Abbreviating $\int_{\mathbb{R}^{3,1}}\int_{\mathbb{R}^{3,1}}\int_{\mathbb{R}^{3,1}}$ to $\iiint$, we obtain
\begin{align*}
f\star_{N}g(x)& =\iiint\lambda_{e_t}(\boldsymbol{w})\mathcal{K}_{-e_t,e_t,e_t}^{M_1}\mathcal{K}_{-i_3,-i_3,i_3}^{M_2}[f_{+}\boldsymbol{y})g_{+}(\boldsymbol{z})]_{+}\\
&+\iiint\lambda_{e_t}(\boldsymbol{w})\mathcal{K}_{-e_t,e_t,-e_t}^{M_1}\mathcal{K}_{-i_3,i_3,-i_3}^{M_2}[f_{+}\boldsymbol{y})g_{+}(\boldsymbol{z})]_{-}\\
&+\iiint\lambda_{e_t}(\boldsymbol{w})\mathcal{K}_{-e_t,e_t,-e_t}^{M_1}\mathcal{K}_{-i_3,-i_3,i_3}^{M_2}[f_{+}\boldsymbol{y})g_{-}(\boldsymbol{z})]_{+}\\
&+\iiint\lambda_{e_t}(\boldsymbol{w})\mathcal{K}_{-e_t,e_t,e_t}^{M_1}\mathcal{K}_{-i_3,i_3,-i_3}^{M_2}[f_{+}\boldsymbol{y})g_{-}(\boldsymbol{z})]_{-}\\
&+\iiint\lambda_{e_t}(\boldsymbol{w})\mathcal{K}_{-e_t,e_t,e_t}^{M_1}\mathcal{K}_{i_3,-i_3,i_3}^{M_2}[f_{-}\boldsymbol{y})g_{+}(\boldsymbol{z})]_{+}\\
&+\iiint\lambda_{e_t}(\boldsymbol{w})\mathcal{K}_{-e_t,e_t,-e_t}^{M_1}\mathcal{K}_{i_3,i_3,-i_3}^{M_2}[f_{-}\boldsymbol{y})g_{+}(\boldsymbol{z})]_{-}\\
&+\iiint\lambda_{e_t}(\boldsymbol{w})\mathcal{K}_{-e_t,e_t,-e_t}^{M_1}\mathcal{K}_{i_3,-i_3,i_3}^{M_2}[f_{-}\boldsymbol{y})g_{-}(\boldsymbol{z})]_{+}\\
&+\iiint\lambda_{e_t}(\boldsymbol{w})\mathcal{K}_{-e_t,e_t,e_t}^{M_1}\mathcal{K}_{i_3,i_3,-i_3}^{M_2}[f_{-}\boldsymbol{y})g_{-}(\boldsymbol{z})]_{-}.\\
\end{align*}

We now only show explicitly how to simplify the second triple integral, the others follow the same pattern.
\begin{align*}
&\int_{\mathbb{R}^{3,1}}\int_{\mathbb{R}^{3,1}}\int_{\mathbb{R}^{3,1}}\lambda_{e_t}(\boldsymbol{w})\mathcal{K}_{-e_t}^{M_1}(w_t,t_1)\mathcal{K}_{e_t}^{M_1}(w_t,t_1')\mathcal{K}_{-e_t}^{M_1}(w_t,t_1'')\\
&\times \mathcal{K}_{-e_t}^{M_2}(\underline{y},\underline{w})\mathcal{K}_{e_t}^{M_2}(\underline{z},\underline{w})\mathcal{K}_{-e_t}^{M_2}(\underline{x},\underline{w})[f_{+}(\boldsymbol{y})g_{+}(\boldsymbol{z})]_{-}d^4\boldsymbol{y}d^4\boldsymbol{z}d^4\boldsymbol{w}\\
&=\int_{\mathbb{R}^{3,1}}\int_{\mathbb{R}^{3,1}}\int_{\mathbb{R}^{3,1}}(2\pi)^2B_1^{\frac{1}{2}}B_2^{\frac{7}{2}}\boldsymbol{e}^{e_t(\frac{D_1w_t^2}{2B_1}+\frac{D_2\underline{w}^2}{2B_2})}\\
&\times (2\pi)^{-\frac{3}{2}}(B_1)^{-\frac{3}{2}}\times \boldsymbol{e}^{\frac{e_t}{2B_1}(A_1(t_1'^2-t_1^2-t_1''^2)-2(t_1'-t_1-t_1'')w_t-D_1w_t^2)}\\
&\times (2\pi)^{-\frac{9}{2}}(B_2)^{-\frac{9}{2}}\times \boldsymbol{e}^{\frac{e_t}{2B_2}(A_2(\underline{z}^2-\underline{x}^2-\underline{y}^2)-2(\underline{z}-\underline{x}-\underline{y})\cdot \underline{w}-D_2\underline{w}^2)}\\
&\times [f_{+}(\boldsymbol{y})g_{+}(\boldsymbol{z})]_{-}d^4\boldsymbol{y}d^4\boldsymbol{z}d^4\boldsymbol{w}\\
&=\int_{\mathbb{R}^{3,1}}\int_{\mathbb{R}^{3,1}}(2\pi)^{-1}(B_1)^{-1}\boldsymbol{e}^{\frac{e_t}{2B_1}[A_1(t_1'^2-t_1^2-t_1''^2)]}\int_{\mathbb{R}}\boldsymbol{e}^{\frac{-e_t}{B_1}(t_1'-t_1-t_1'')w_t}dw_t\\
&\times (2\pi)^{-3}(B_2)^{-1}\boldsymbol{e}^{\frac{e_t}{2B_2}[A_2(\underline{z}^2-\underline{x}^2-\underline{y}^2)]}\int_{\mathbb{R}^{3}}\boldsymbol{e}^{\frac{-e_t}{B_2}(\underline{z}-\underline{x}-\underline{y})\cdot \underline{w}}d\underline{w}\\
&\times [f_{+}(\boldsymbol{y})g_{+}(\boldsymbol{z})]_{-}d^4\boldsymbol{y}d^4\boldsymbol{z}d^4\boldsymbol{w}\\
&=\int_{\mathbb{R}^{3,1}}\int_{\mathbb{R}^{3,1}}\boldsymbol{e}^{\frac{e_t}{2B_1}[A_1(t_1'^2-t_1^2-t_1''^2)]}\delta(t_1'-t_1-t_1'')\\
&\times \boldsymbol{e}^{\frac{e_t}{2B_2}[A_2(\underline{z}^2-\underline{x}^2-\underline{y}^2)]}\delta(\underline{z}-\underline{x}-\underline{y})d^4\boldsymbol{y}d^4\boldsymbol{z}\\
&=\int_{\mathbb{R}^{3,1}}[f_{+}(\boldsymbol{y})g_{+}(-(t_1-t_1'),\underline{x}+\underline{y})]_{-}d^4\boldsymbol{y}\\
&=\int_{\mathbb{R}^{3,1}}[f_{+}(\boldsymbol{y})g_{+}^{(1,1)}(t_1-t_1',\underline{x}+\underline{y})]_{-}d^4\boldsymbol{y}\\
&=[f_{+}\star g_{+}^{(1,1)}(t_1,-\underline{x})]_{-}.\\
\end{align*}

Simplifying the other seven triple integrals in similar manner, we obtain the the desired decomposition for the proposed convolution. This completes the proof of Theorem \ref{thm5.2}.
\end{proof}

Now we state and prove our new convolution theorem.
\begin{theorem}\label{thm5.3}
Let $p(\boldsymbol{x})=(f\star_N g)(\boldsymbol{x})$ and $\mathcal{L}_{A}[p](\boldsymbol{w})$, $\mathcal{L}_{A}[f](\boldsymbol{w})$ and $\mathcal{L}_{A}[g](\boldsymbol{w})$ denote the two-sided LCST of $p(\boldsymbol{x})$, $f(\boldsymbol{x})$ and $g(\boldsymbol{x})$. Then
\begin{equation}
\mathcal{L}_{A}[p(\boldsymbol{x})](\boldsymbol{w})=\lambda_{e_t}(\boldsymbol{w})\mathcal{L}_{A}[f(\boldsymbol{x})](\boldsymbol{w})\mathcal{L}_{A}[g(\boldsymbol{x})](\boldsymbol{w}).\label{eq5.11}
\end{equation}
\end{theorem}

\begin{proof}
\begin{equation}
\begin{aligned}
(f\star_Ng)(\boldsymbol{x})&=\int_{\mathbb{R}^{3,1}}\int_{\mathbb{R}^{3,1}}\int_{\mathbb{R}^{3,1}}\lambda_{e_t}(\boldsymbol{w})\mathcal{K}_{-e_t}^{M_{1}}(w_{t},t_1)\mathcal{K}_{e_t}^{M_{1}}(w_{t},t_1')f(\boldsymbol{y})\mathcal{K}_{i_3}^{M_{2}}(\underline{y},\underline{w}) \\
&\mathcal{K}_{e_t}^{M_{1}}(w_{t},t_1'')\mathcal{K}_{i_3}^{M_{2}}(\underline{z},\underline{w})\mathcal{K}_{-i_3}^{M_{2}}(\underline{x},\underline{w})d^4\boldsymbol{y}d^4\boldsymbol{z}d^4\boldsymbol{w}\\
&=\int_{\mathbb{R}^{3,1}}\lambda_{e_t}(\boldsymbol{w})\mathcal{K}_{-e_t}^{M_{1}}(w_{t},t_1)\int_{\mathbb{R}^{3,1}}\mathcal{K}_{e_t}^{M_{1}}(w_{t},t_1')f(\boldsymbol{y})\mathcal{K}_{i_3}^{M_{2}}(\underline{y},\underline{w})d^4\boldsymbol{y}\\
&\times \int_{\mathbb{R}^{3,1}}\mathcal{K}_{e_t}^{M_{1}}(w_{t},t_1'')\mathcal{K}_{i_3}^{M_{2}}(\underline{z},\underline{w})d^4\boldsymbol{z}\mathcal{K}_{-i_3}^{M_{2}}(\underline{x},\underline{w})d^4\boldsymbol{w}\\
&=\int_{\mathbb{R}^{3,1}}\mathcal{K}_{-e_t}^{M_{1}}(w_{t},t_1)\lambda_{e_t}(\boldsymbol{w})\mathcal{L}_{A}[f(\boldsymbol{y})](\boldsymbol{w})\mathcal{L}_{A}[g(\boldsymbol{z})](\boldsymbol{w})\mathcal{K}_{-i_3}^{M_{2}}(\underline{x},\underline{w})d^4\boldsymbol{w}\\
&=\mathcal{L}_{A}^{-1}[\lambda_{e_t}(\boldsymbol{w})\mathcal{L}_{A}[f(\boldsymbol{x})](\boldsymbol{w})\mathcal{L}_{A}[g(\boldsymbol{x})](\boldsymbol{w})].\\
\end{aligned}\label{eq5.12}
\end{equation}
\end{proof}

\section{Conclusion}
\label{sec6}

After an overview of space-time algebra and considering the fractional space-time Fourier transform (FrSFT) and some of its properties, we propose the concept of linear canonical space-time transform (LCST) as a broad generalization of FrSFT. Finally, we propose a convolution operator and the corresponding convolution theorem. We hope that applied mathematics will continue to investigate this new class of hyper-complex transformations.

\section*{Acknowledgments}
This work was supported by the National Natural Science Foundation of China [No. 62171041].

\section*{Declarations}
The authors declare that they have no known competing financial interests or personal relationships that could have appeared to influence the work reported in this paper.

\section*{Data Availability Statement}
Not applicable.

\section*{Code Availability Statement}
Not applicable.

\bibliography{sn-bibliography}

\end{document}